\documentclass[11pt]{article}

\usepackage{amsmath, amssymb, amsfonts, amsopn, amsthm}
\usepackage{graphicx}
\usepackage{subfigure}
\usepackage{etoolbox} 
\usepackage[percent]{overpic}
\usepackage[parfill]{parskip}
\usepackage{tikz}
\usepackage{stmaryrd}
\usepackage{enumerate}

\usepackage{booktabs}

\DeclareMathOperator{\rk}{rank}
\DeclareMathOperator{\kernel}{Ker}
\DeclareMathOperator{\im}{Im}
\DeclareMathOperator{\sign}{sign}
\DeclareMathOperator{\id}{\id}

\newtheorem{thm}{Theorem}[section]

\newtheorem{lemma}[thm]{Lemma}

\numberwithin{equation}{section}
\newtheorem{example}[thm]{Example}

\newcommand{\Z}{\mathbb{Z}}
\newcommand{\R}{\mathbb{R}}
\newcommand{\I}{\mathcal{I}}

\newcommand{\rp}{\mathbb{R}\text{P}}
\newcommand{\rpt}{\mathbb{R}\text{P}^3}
\newcommand{\ijk}{{i,j,k}}
\newcommand{\jk}{{j,k}}
\newcommand{\lra}{\longrightarrow}
\newcommand{\lms}{\longmapsto}
\newcommand{\kbsm}{\mathcal{S}_{2,\infty}}
\newcommand{\poincare}{\mathcal{P}}
\newcommand{\X}{\mathcal{X}}
\newcommand{\bi}{\overline{1}}
\newcommand{\bX}{\overline{X}}
\newcommand{\bV}{\overline{V}}
\newcommand{\bI}{\overline{I}}

\newcommand{\shift}[1]{\lbrace #1 \rbrace}
\newcommand{\Ri}{\textup{R-I}}
\newcommand{\Rii}{\textup{R-II}}
\newcommand{\Riii}{\textup{R-III}}
\newcommand{\Riv}{\textup{R-IV}}
\newcommand{\Rv}{\textup{R-V}}

\graphicspath{{./images/}}
\newcommand{\brac}[1] {\left\llbracket \raisebox{-0.63 em}{\includegraphics{#1}} \right\rrbracket}
\newcommand{\bracmid}[1] {\left\llbracket \raisebox{-0.53 em}{\includegraphics{#1}} \right\rrbracket}
\newcommand{\bracsmall}[1] {\left\llbracket \raisebox{-0.43 em}{\includegraphics[scale=0.75]{#1}} \right\rrbracket}
\newcommand{\bracsmaller}[1] {\left\llbracket \raisebox{-0.33 em}{\includegraphics[scale=0.75]{#1}} \right\rrbracket}
\newcommand{\smallbullet}{\,\begin{picture}(-1,1)(-1,-3)\circle*{2.0}\end{picture} \; \, }

\date{2013}
\begin{document}
\title{The categorification of the Kauffman bracket skein module of $\rpt$}

\author{Bo\v stjan Gabrov\v sek}

\maketitle


\begin{abstract}
Khovanov homology, an invariant of links in $\R^3$, is a graded homology theory that categorifies the Jones polynomial in the sense that the graded Euler characteristic of the homology is the Jones polynomial. Asaeda, Przytycki and Sikora generalized this construction by defining a double graded homology theory that categorifies the Kauffman bracket skein module of links in $I$-bundles over surfaces, except for the surface $\rp^2$, where the construction fails due to strange behaviour of links when projected to the non-orientable surface $\rp^2$. This paper categorifies the missing case of the twisted $I$-bundle over $\rp^2$, $\rp^2 \widetilde{\times} I \approx \rpt \setminus \{\ast\}$, 
by redefining the differential in the Khovanov chain complex in a suitable manner.
\end{abstract}


\section{Introduction}

In 1987 Przytycki and Turaev introduced the study of skein modules \cite{Pr2, Tu}. The most studied skein module, the Kauffman bracket skein module (KBSM), is a generalization of the Kauffman bracket polynomial, which itself is a reformulation of the Jones polynomial. Since then, the KBSM has been calculated for a number of different $3$-manifolds and is a powerful  invariant of framed links in these manifolds \cite{Pr1, HP2, Mr, MD}.

Incidentally, in 1990 Yu. V. Drobotukhina introduced the study of links in the real projective space by providing an invariant of such links, a version of the Jones polynomial for $\rpt$ \cite{Dr}.

A major breakthrough in the study of knots in $\R^3$ appeared in the late 1990s by a series of lectures by M. Khovanov, who managed to categorify the Jones polynomial by constructing a chain complex of graded vector spaces with the property that the homology of this chain complex, the Khovanov homology, is a link invariant. Moreover, the graded Euler characteristic of this complex is the Jones polynomial \cite{Kh}.
Perhaps the most outstanding consequence of this theory is the $s$-invariant of Rasmussen, which gives a bound on a knot's slice genus and is sufficient to prove the Milnor conjecture. More recently, Kronheimer and Mrowka showed that the Khovanov homology detects the unknot \cite{KM}.

Various generalizations of the Khovanov homology have been constructed so far. One example is that the KBSM has been categorified for $I$-bundles over all surfaces except $\rp^2$ \cite{APS}. Another generalization was done  by Manturov, who managed to categorify the Jones polynomial of virtual links \cite{Ma}, which, as a special case, includes the categorification of the Jones polynomial of links in $\rpt$.

This paper provides details and explicitly shows how to categorify the KBSM of links in $\rpt$, which is equivalent to categorifying the KBSM of the twisted $I$-bundle over $\rp^2$, the missing piece of the puzzle in \cite{APS}.

\section{The Kauffman bracket skein module of $\boldsymbol{\rpt}$}
\label{sec:kbsm}
To have a working theory of links in $\rpt$, we must first introduce suitable diagrams of these links, which we call \emph{projective links}.
By identifying $\rpt \setminus \{\ast\} \approx \rp^2 \widetilde{\times} I$, the twisted $I$-bundle over $\rp^2 $, a link can be projected to $\rp^2$, a $2$-disk with antipodal points identified on its boundary. Such  diagrams are accompanied by five Reidemeister moves: the three classical Reidemeister moves $\Ri -  \Riii$  and two additional moves $\Riv$ and $\Rv$ that act across the boundary of the $2$-disk (Fig. \ref{fig:reid}) \cite{Dr}. Two links are ambient isotopic in $\rpt$ if a diagram of one link can be transformed into a diagram of the other by a finite sequence of Raidemeister moves $\Ri$ - $\Rv$.

\begin{figure}[tbh]
\centering
\subfigure[$\Ri$]{
\includegraphics[]{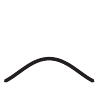} \raisebox{1.6\height}{$ \longleftrightarrow $} \includegraphics[]{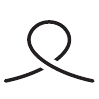}
}
\hspace{6 mm}
\subfigure[$\Rii$]{
\includegraphics[]{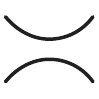} \raisebox{2.6\height}{$ \longleftrightarrow $} \includegraphics[]{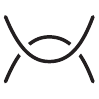}
}
\hspace{6 mm}
\subfigure[$\Riii$]{
\includegraphics[]{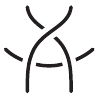} \raisebox{2.6\height}{$ \longleftrightarrow $} \includegraphics[]{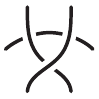}
}

\subfigure[$\Riv$]{
\includegraphics[]{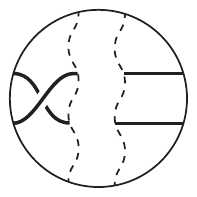} \raisebox{5.6\height}{$ \longleftrightarrow $} \includegraphics[]{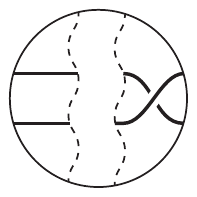}
}
\hspace{10 mm}
\subfigure[$\Rv$]{
\includegraphics[]{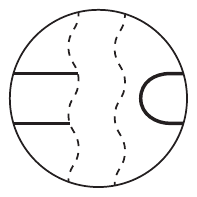} \raisebox{5.6\height}{$ \longleftrightarrow $} \includegraphics[]{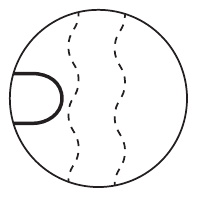}
}
\caption{Three classical and two additional Reidemeister moves}
\label{fig:reid}
\end{figure}

The Kauffman bracket skein module of a 3-manifold $M$ is constructed as follows. Take a coefficient ring $R$ with a distinguished unit $A \in R$. Let $\mathcal{L}_{fr}(M)$ be the set of isotopy classes of framed links in $M$, including the class of the empty link $[\emptyset]$,  and let $R\mathcal{L}_{fr}(M)$ be the free $R$-module spanned by $\mathcal{L}(M)$.

As in the case of the Jones polynomial, we would like to impose the skein relation and the framing relation in $R\mathcal{L}_{fr}(M)$, we therefore take the submodule $\mathcal{S}_{fr}(M)$ of $R\mathcal{L}_{fr}(M)$ generated by
\begin{align*}
\tag{skein relation} \raisebox{-4.5pt}{\includegraphics{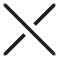}} - A \raisebox{-4.5pt}{\includegraphics{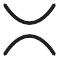}} -A^{-1} \raisebox{-4.5pt}{\includegraphics{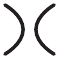}},\\
\tag{framing relation} L \sqcup \raisebox{-4.5pt}{\includegraphics{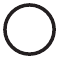}} - (-A^2 - A^{-2}) L.
\end{align*}

The Kauffman bracket skein module $\kbsm(M)$ is $R\mathcal{L}_{fr}(M)$ modulo these two relations: $$\kbsm(M) = R\mathcal{L}_{fr}(M) / \mathcal{S}_{fr}(M).$$

It is shown in \cite{HP1} that the Kauffman bracket skein module of the lens space $L(p,q)$ is free with $\lfloor \frac{p}{2} \rfloor + 1$ generators, in particular, $\kbsm(\rpt)$ has two generators, the isotopy class of the empty set $[\emptyset]$ and the class of the orientation reversing curve x of $\rp^2 \subset \rpt$, shown in Fig. \ref{fig:rptgen}.

\begin{figure}[htb]
\centering
\begin{overpic}{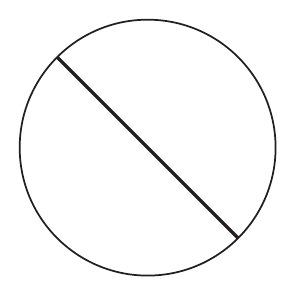}\put(53,53){$x$}\end{overpic}
\caption{A generator of the KBSM of $\rpt$}
\label{fig:rptgen}
\end{figure}

By expressing a link $L \in \rpt$ in terms of the two generators and setting $[\emptyset] = (-A^2 - A^{-2})^{-1}$ and $[x] = 1$, we get the Kauffman bracket polynomial $\langle L \rangle$ of $L$. By further normalizing the Kauffman bracket by $(-A^3)^{-w(L)}$ we get the Laurent polynomial $X(L) = (-A^3)^{-w(L)} \langle L \rangle$, where $w(L)$ stands for the \emph{writhe} of $L$. By substituting $A=t^{-\frac{1}{4}}$ in $X(L)$ we get the Jones polynomial in $\rpt$ described in \cite{Dr}. Due to the normalization, the Jones polynomial is an invariant of unframed links in $\rpt$. 

We continue by providing a state-sum formula for calculating the KBSM of a projective link.
Let $D$ be an oriented projective link diagram with $n$ crossings.
First, order the crossings arbitrarily from $1$ to $n$ and denote the set of crossings by  $\X$.
Assign each crossing a $+$ or $-$ \emph{sign} using the right-hand rule in Fig. \ref{fig:csign}. The number of positive crossings is marked by $n_+$ and the number of negative crossings is marked by $n_-$.

\begin{figure}[tbh]
\centering
\subfigure[$\sign = +1$]{
\hspace{5 mm} \includegraphics[]{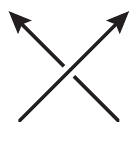} \hspace{5 mm}
}
\subfigure[$\sign = -1$]{
\hspace{5 mm} \includegraphics[]{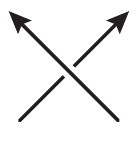} \hspace{5 mm}
}
\caption{The sign of a crossing}
\label{fig:csign}
\end{figure}

Each crossing can be smoothened by a smoothening of type $0$ or $1$ (Fig. \ref{fig:smoothings}). We call $\{ 0, 1 \}^\X$ the \emph{discrete cube} of $D$ and a vertex $s \in \{ 0, 1\}^\X$ a \emph{(Kauffman) state} of $D$. Each state corresponds to a diagram with all crossings smoothened either by a type $0$ or a type $1$ smoothening. For convenience, this complete smoothening is also called a \emph{state} of $D$.
Each state is just a collection of disjoint closed loops which are called \textit{circles}. A circle is \emph{trivial} if it bounds a disk in $\rp^2$ and \emph{projective} if it does not.
By examining  the parity of the number of components crossing the boundary of the projecting disk in $\rp^2$, it can be easily deduced that there is at most one projective circle in each state.
\begin{figure}[tbh]
\centering
\subfigure[type $0$ smoothening]{
\includegraphics[]{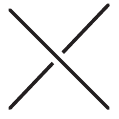} \raisebox{3.4\height}{$ \lra $} \includegraphics[]{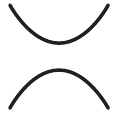}
}
\hspace{10 mm}
\subfigure[type $1$ smoothening]{
\includegraphics[]{skein-x} \raisebox{3.4\height}{$ \lra $} \includegraphics[]{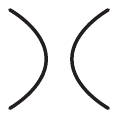}
}
\caption{Two types of smoothenings}
\label{fig:smoothings}
\end{figure}

For a state $s$, we denote by $|s|_T$ the number of trivial circles and by $|s|_P$ the number of projective circles, the number of all circles is denoted by $|s| = |s|_T + |s|_P$.
By $\#0(s)$ we denote the number of $0$ factors in $s$ and by $\#1(s)$ the number of $1$ factors.

The state-sum formula for calculating $[ L ] \in \kbsm(\rpt)$ in terms of the standard generators is
$$\kbsm(\rpt)([L]) = \sum_{s \in \{ 0, 1 \}^\X} A^{\#0(s) - \#1(s)} (-A^2 - A^{-2})^{|s|_T}[x]^{|s|_P}[\emptyset]^{1-|s|_P},$$
where we often use the normalization $[\emptyset] = 1$.

\section{The chain complex}
\label{sec:cx}

For a link $L \subset \R^3$ Khovanov managed to construct a chain complex of graded vector spaces that categorifies the Jones polynomial in the sense that the graded Euler characteristic of that complex is exactly the Jones polynomial \cite{Kh}.
The homology of the chain complex turns out to be a link invariant stronger than the Jones polynomial itself.
Asaeda, Przytycki and Sikora managed to construct a chain complex of bigraded vector spaces that categorify the KBSM of $I$-bundles over surfaces \cite{APS}.
This construction did not work for the twisted $I$-bundle over $\rp^2$, the problem is in essence the strange behaviour of links projected to $\rp^2$, more precisely, the $1 \rightarrow 1$ bifurcations which will be described latter.
Manturov managed to overcome the problem of $1 \rightarrow 1$ bifurcations when he categorified the Jones polynomial of virtual links \cite{Ma}.

In this section we define the Khovanov chain complex with gradings similar to those defined by Asaeda, Przytycki and Sikora. But in the next section, we define the differential using  techniques of Manturov to control the odd behaviour of circles mentioned above.

Let $W = \bigoplus_\jk W_\jk$ be a $\Z \oplus \Z$-graded $\Z$-module. The \emph{Poincar\' e polynomial}, in variables $A$ and $z$, of $W$ is defined as
$$\poincare(W)=\sum_\jk A^j z^k \rk W_\jk.$$
In the standard construction of the Khovanov homology the tensor product is used to form chain complexes, we, however, have to use the wedge product (which is in some sense an ordered version of the tensor product).

The \emph{wedge product} $W \wedge W'$ of the modules $W = \bigoplus_\jk W_\jk$ and $W' = \bigoplus_\jk W'_\jk$ is defined as
$$W \wedge W' = \bigoplus_{\substack{j=j_1+j_2\\k=k_1+k_2}}W_{j_1,k_1} \wedge W'_{j_2,k_2}.$$
The wedge product is associative and anticommutative.

For a permutation $\sigma \in S_n$, the wedge product of  $n$ modules $W_1, W_2, \dotsc, W_n$ is subject to the \emph{permutation rule}:
$$W_{\sigma(1)} \wedge W_{\sigma(2)} \wedge \cdots \wedge W_{\sigma(n)} = \sign(\sigma) \, W_1 \wedge W_2 \wedge \cdots \wedge W_n.$$
The   \emph{degree shift} operator $\smallbullet \shift{l,m}$ shifts the gradings of $W = \bigoplus_\jk W_\jk$  by $(l,m)$:
$$W\shift{l,m} = \bigoplus_\jk W_{j-l,k-m}.$$ 
The degree shift $\smallbullet \shift{l,0}$ is abbreviated to $\smallbullet \shift{l}$, such a shift corresponds to multiplication by $A^l$ in the Poincar\' e polynomial: $\poincare(W \shift{l}) = A^l \poincare(W).$

Let $V = \langle 1,X \rangle$ be the bigraded $\Z$-module freely generated by elements $1$ and $X$ with bigradings $\deg 1  = (-2,0)$ and $\deg X =(2,0)$,
and let $\bV=\langle \bi, \bX \rangle$ the module generated by $\bi$ and $\bX$ with $\deg \bi  = (0,1)$ and $\deg \bX =(0,-1)$.
Note that $\poincare (V) = A^2 + A^{-2} $ and $\poincare (\bV) = z + z^{-1}$, which is exactly what we assign a circle (trivial or resp. projective) in the state-sum formula of the KBSM.

Using the notation in section \ref{sec:kbsm} and assuming the circles are enumerated in a way that the possible projective circle is at the end,
$C_s$ will represent the module associated with the state $s$. A circle in $s$, enumerated with $i$, contributes a factor $V_i$ if the circle is trivial or a factor $\bV_i$ if the circle is projective.
The module $C_s$ is therefore either equal to
$$C_s = (V_1 \wedge V_2 \wedge \cdots \wedge V_n) \shift{\#0(s) - \#1(s)}$$
or is equal to
$$C_s = (V_1 \wedge V_2 \wedge \cdots \wedge V_{n-1} \wedge \bV_n)\shift{\#0(s) - \#1(s)},$$
depending whether $s$ has a projective circle or not. 

The module of $i$-chains is now defined to be the sum of all $C_s$'s where the difference between the number of $0$ and $1$ factors in $s$ is $i$:
$$C^i = \bigoplus_{\substack{s \in \{ 0, 1\}^\X\\ \#0(s) - \#1(s) = i}} C_s.$$
Forming such direct sums of certain modules is sometimes called a \emph{flattening} of $\{C_s\}_s.$
We claim that
$$0 \lra C^{n} \lra C^{n-2} \lra \cdots \lra C^{-n+2} \lra C^{-n} \lra 0$$
with suitable differentials yet to be defined, forms a Khovanov chain complex.

\section{The differential}
\label{sec:diff}

Let $\alpha \in \{ 0, 1, \star \}^\X$ be a sequence with the property that $\star$ appears only once in $\alpha$.
By replacing $\star$ by $0$ we get a vertex of the discrete cube $\lbrace 0, 1\rbrace^\X$ denoted by $\alpha_{\star \rightarrow 0}$
and by replacing $\star$ by $1$ we get an adjacent vertex $\alpha_{\star \rightarrow 1}$.
We call such an $\alpha$ an \emph{edge} of $\{ 0, 1\}^\X$.
Recall that each vertex corresponds to a state of a diagram, that being so, each edge corresponds to a local "change" of two adjacent states.
Each $\alpha$ is associated with a linear map $d_\alpha: C_{\alpha_{\star \rightarrow 0}} \lra C_{\alpha_{\star \rightarrow 1}}$ called a \emph{partial differential}, which will be defined in the next few paragraphs.

For a diagram $D$, we call the collection of modules $\lbrace C_s \rbrace_{s \in \lbrace 0, 1\rbrace^\X}$, together with their partial differentials, the \emph{cube} $\left\llbracket D \right\rrbracket$ of $D$. Now, $\left\llbracket D \right\rrbracket$ needs to form a anticommutative diagram, since this forces the flattened $\left\llbracket D \right\rrbracket$ to form a well-defined chain complex $(C_\bullet, d_\bullet)$.
As seen in \cite{APS}, there is no obvious way to define the partial differential for links in $\rpt$,
but as shown in \cite{Ma} this is possible to achieve for the Jones polynomial, if the links and the circles in the states are oriented and certain signs are applied to the partial differentials.

Hence, we orient both, the diagram $D$ and the circles of the states in an arbitrary manner.
In the neighbourhood of a crossing there are either one or two circles. If we change the smoothening in a crossing from type $0$ to type $1$, we call the change of the involved circles a \emph{bifurcation}.
It is evident that the circles involved in the states $\alpha_{\star \rightarrow 0}$ and $\alpha_{\star \rightarrow 1}$ are in 1-1 correspondence, with one of these exceptions (Fig. \ref{fig:bif}):
\begin{enumerate}[(a)]
\item two circles in  $\alpha_{\star \rightarrow 0}$ join into one circle in  $\alpha_{\star \rightarrow 1}$ (type $2 \rightarrow 1$ bifurcation),
\item a circle in  $\alpha_{\star \rightarrow 0}$ splits into two circles in  $\alpha_{\star \rightarrow 1}$ (type $1\rightarrow 2$ bifurcation),
\item a circle in $\alpha_{\star \rightarrow 0}$ twists into a circle in  $\alpha_{\star \rightarrow 1}$ (type $1\rightarrow 1$ bifurcation).
\end{enumerate}

\begin{figure}[tbh]
\centering
\subfigure[type $2 \rightarrow 1$]{
\includegraphics[]{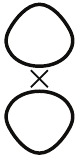} \raisebox{4.3\height}{$ \lra $} \includegraphics[]{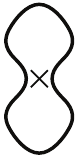}
}
\hspace{3 mm}
\subfigure[type $1\rightarrow 2$]{
\includegraphics[]{biff-1} \raisebox{4.3\height}{$ \lra$} \includegraphics[]{biff-2}
}
\hspace{3 mm}
\subfigure[type $1\rightarrow 1$]{
\raisebox{0.15 \height} { \includegraphics[]{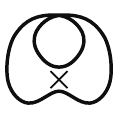} \raisebox{2.8\height}{$ \lra$} \includegraphics[]{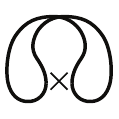}}
}
\caption{Three types of bifurcations}
\label{fig:bif}
\end{figure}

Note that the type $1\rightarrow 1$ bifurcation can only appear in the case of projecting to a non-orientable surface (such as $\rp^2$).

Looking at a crossing $c$ of $D$ in such a way that the two outgoing arcs are facing northwest and northeast and that the ingoing arcs are facing southwest and southeast, we call a circle of some state \emph{locally consistently oriented} at $c$ if its orientation agrees with the orientation of the northeast arc or disagrees with the orientation of the southwest arc (with respect to the arc before the smoothening); vice versa, a circle is \emph{locally inconsistently oriented} at $c$ if its orientation disagrees with the northeast arc or agrees with the southwest arc (see Fig. \ref{fig:proper}). It can happen that both arcs in the state belong to the same circle and the consistency cannot be uniquely defined, in this case the local consistency is \emph{indetermined} at $c$.

\begin{figure}[tbh]
\centering
\begin{tabular}{ccc} \includegraphics[]{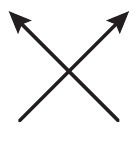}  & \includegraphics[]{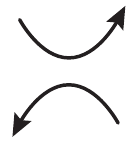} & \includegraphics[]{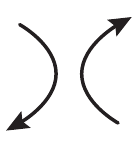} \end{tabular}
\caption{Locally consistent orientations}
\label{fig:proper}
\end{figure}

In order to define the differential, we rearrange the $\wedge$-factors of $C_{\alpha_{\star \rightarrow 0}}$ by a permutation $\sigma \in S_{|\alpha_{\star \rightarrow 0}|}$ in such a way that the factors involved in the bifurcation are at the beginning of the wedge product.
Furthermore, if the bifurcation is of type $2 \rightarrow 1$, assuming that the arcs of the diagram are facing northwest and northeast, we wish that the first factor in the domain is represented by a "left" circle when the bifurcation site is at a positive crossing and is represented by a "top" circle when the bifurcation site is at a negative crossing (see Fig. \ref{fig:difforient}).

\begin{figure}[tbh]
\centering
\subfigure[positive crossing]{ \begin{tabular}{c}
\includegraphics{positive}  \\ 
\begin{overpic}{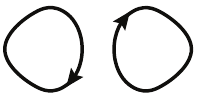} \put(18,18.5){$1$} \put(72,18.5){$2$}\end{overpic} \raisebox{0.4 cm}{$\overset{m}{\lra}$}  \includegraphics{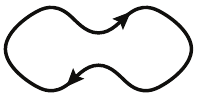} \\
\includegraphics{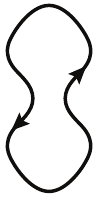} \raisebox{0.9 cm}{$\overset{\Delta}{\lra}$}  \begin{overpic}{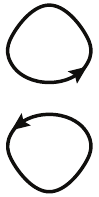} \put(20,18){$2$} \put(20,70){$1$}\end{overpic} \\
\end{tabular}}
\subfigure[negative crossing]{ \begin{tabular}{c}
\includegraphics{negative}  \\ 
\begin{overpic}{diff-vs} \put(20,18){$2$} \put(20,70){$1$}\end{overpic} \raisebox{0.9 cm}{$\overset{m}{\lra}$}  \includegraphics{diff-vj} \\
\includegraphics{diff-hj} \raisebox{0.4 cm}{$\overset{\Delta}{\lra}$}  \begin{overpic}{diff-hs} \put(18,18.5){$1$} \put(72,18.5){$2$}\end{overpic} \\
\end{tabular}}
\caption{Order of (locally consistently oriented) circles of a bifurcation}
\label{fig:difforient}
\end{figure}

If the bifurcation is of type $2 \rightarrow 1$, we \emph{multiply} the first two factors;
if the  bifurcation is of type $1\rightarrow 2$, we \emph{comultiply} the first factor and 
if the bifurcation is of type $1\rightarrow 1$ we apply the $0$ map to the first factor.
On the remaining factors we apply the semi-identity $\I = I \wedge I \wedge \cdots \wedge I \, ( \wedge \, \bI),$
where the map $I:V\lra V$  (resp. $\bI: \bV \lra \bV$) is defined by
$I(1) = 1, I(X) = \pm X$ (resp. $\bI(\bi) = \bi, \bI(\bX) = \pm \bX$), with a plus sign on the generator $X$ if the circles that two $V$'s (resp. $\bV$'s) represent have the same orientations and with a minus sign if the orientations are opposite.

Let $m$ be the \emph{multiplication operator} and $\Delta$ the \emph{comultiplication operator}, they are both linear maps
subject to the rules in Table \ref{table:mc}.

\begin{table}
\caption{Multiplication and comultiplication (shifts omitted)} \label{table:mc}
\centering
\begin{tabular}{| c@{\,}c@{\,}c@{\,}c@{\,}c@{\,}c | c@{\,}  c@{\,}c@{\,}c@{\,}c@{\,}c@{\,}c |}
\toprule &&&&&&&&&&&& \\[-0.5 em]
$m:$ & $V_1$  & $\wedge$ & $V_2$           & $\lra$  & $V$           & $\Delta:$ & $V$            & $\lra$   & $V_1$          & $\wedge$ & $V_2$ & \\ 
            & $1_1$  & $\wedge$ & $1_2$            & $\lms$ & $1$            &                    &  $1$           & $\lms$ & $1_1$            & $\wedge$ & $(\pm X_2)$ & $+ \, (\pm X_1) \wedge 1_2$ \\
            & $1_1$  & $\wedge$ & $(\pm X_2)$ & $\lms$ & $(\pm X)$ &                    &  $(\pm X)$ & $\lms$ & $(\pm  X_1)$ & $\wedge$ & $(\pm X_2)$ & \\
  & $(\pm X_1)$ & $\wedge$ & $1_2$            & $\lms$ & $(\pm X)$ &&&&&&& \\
  & $(\pm X_1)$ & $\wedge$ & $(\pm X_2)$ & $\lms$ & $0$  &&&&&&&\\[0.5 em] \midrule &&&&&&&&&&&& \\[-0.5 em]
  $m:$ & $\bV_1$  & $\wedge$ & $V_2$           & $\lra$  & $\bV$ & $\Delta:$ & $\bV$            & $\lra$   & $\bV_1$          & $\wedge$ & $V_2$ & \\
            & $\bi_1$  & $\wedge$ & $1_2$            & $\lms$ & $\bi$     &                    &  $\bi$           & $\lms$ & $\bi_1$            & $\wedge$ & $(\pm X_2)$ & \\
            & $\bi_1$  & $\wedge$ & $(\pm X_2)$ & $\lms$ & $0$       &                    &  $(\pm \bX)$ & $\lms$ & $(\pm  \bX_1)$ & $\wedge$ & $(\pm X_2)$ & \\
  & $(\pm \bX_1)$ & $\wedge$ & $1_2$            & $\lms$ & $(\pm \bX)$ &&&&&&& \\
  & $(\pm \bX_1)$ & $\wedge$ & $(\pm X_2)$ & $\lms$ & $0$   &&&&&&&\\[0.5 em] \midrule &&&&&&&&&&&& \\[-0.5 em]
  $m:$ & $V_1$  & $\wedge$ & $\bV_2$           & $\lra$  & $\bV$ & $\Delta:$ & $\bV$            & $\lra$   & $V_1$          & $\wedge$ & $\bV_2$ & \\
            & $1_1$  & $\wedge$ & $\bi_2$            & $\lms$ & $\bi$     &                    &  $\bi$           & $\lms$ & $(\pm X_1)$  & $ \wedge$ & $\bi_2$ & \\
            & $1_1$  & $\wedge$ & $(\pm \bX_2)$ & $\lms$ & $(\pm \bX)$ &           &  $(\pm \bX)$ & $\lms$ & $(\pm  X_1)$ & $\wedge$ & $(\pm \bX_2)$ & \\
  & $(\pm X_1)$ & $\wedge$ & $\bi_2$            & $\lms$ & $0$ &&&&&&& \\
  & $(\pm X_1)$ & $\wedge$ & $(\pm \bX_2)$ & $\lms$ & $0$ &&&&&&& \\[0.5 em] \bottomrule
\end{tabular}
\end{table}

As suggested in Fig. \ref{fig:difforient}, the order of the factors in the codomain of $\Delta$ depend on the position of the circles they represent. At a positive crossing the first factor is presented by the "top" circle and at a negative crossing the first factor is represented by a "left" circle.

The sign of each factor $X$ (resp. $\bX$) in Table \ref{table:mc} depends on the local consistency of the circle that $X$ (resp. $\bX$) represents at the crossing where the bifurcation appears. If the circle is locally consistently oriented at the crossing, the sign is positive and if the circle is locally inconsistently oriented, the sign is negative. Note that the consistency is indetermined only at bifurcations of type $1\rightarrow 1$. 

After (co)multiplying and applying the identity, we rearrange the factors in the result by the permutation $\rho \in S_{\alpha_{\star \rightarrow 1}}$, so that they agree with the order of factors in the codomain $C_{\alpha_{\star \rightarrow 1}}$.

The above steps describe the partial differential  $d_\alpha: C_{\alpha_{\star \rightarrow 0}} \lra C_{\alpha_{\star \rightarrow 1}}$.
In detail, if $P_\sigma$ is a map that rearranges factors in $C_{\alpha_{\star \rightarrow 0}}$ by $\sigma$ and $P_\rho$ is the map that rearranges factors in $C_{\alpha_{\star \rightarrow 1}}$ by $\rho$,
\begin{equation}
d_\alpha = \begin{cases} P_\rho \circ (m \wedge \I) \circ P_\sigma, & \mbox{if } |\alpha_{\star \rightarrow 0}| = |\alpha_{\star \rightarrow 1}|+1, \\ P_\rho \circ (\Delta \wedge \I) \circ P_\sigma, & \mbox{if }  |\alpha_{\star \rightarrow 0}|+1 = |\alpha_{\star \rightarrow 1}|, \\ P_\rho \circ (0 \wedge \I) \circ P_\sigma, & \mbox{if } |\alpha_{\star \rightarrow 0}| = |\alpha_{\star \rightarrow 1}|.\end{cases}
\label{eq:partial}
\end{equation}
 
The \emph{total differential} $d^{(i)} : C^i \lra C^{i-2}$ is the sum of partial differentials:
$$d^{(i)} = \bigoplus_{\substack{\alpha \in \{ 0, 1, \star \}^\X\\ \#0(\alpha) - \#1(\alpha) = i-1}} d_\alpha.$$
We shall also call such a sum a \emph{flattening} of partial differentials.

As an illustration of how the partial and total differentials work, a detailed example is presented at the end of the next section.

\section{The homology}
\label{sec:homology}

The \emph{$i$-th Khovanov homology group} is
$$H_i = \frac{\kernel d^{(i)}}{\im d^{(i+2)}}.$$

\begin{thm}[proof in section \ref{sec:proofs}] 
Let $L$ be a link in $\rpt$ and $D$ a diagram of $L$. Then $d \circ d=0$, hence $(C_\bullet, d_\bullet)$ is a chain complex. 
\label{thm:dd}
\end{thm}

\begin{thm}[proof in section \ref{sec:proofs}]
For a diagram $D$ of the link $L \subset \rpt$, $H_\bullet(D)$ is preserved under $\Rii$, $\Riii$, $\Riv$ and $\Rv$. $H_\bullet$ is therefore an invariant of framed links in $\rpt$.
\label{thm:invariance}
\end{thm}

The \emph{Euler characteristic} of a $\Z ^3$-graded $\Z$-module $W = \bigoplus_\ijk W_{j,k}^i$ is for our purpose defined as
$$\chi (W) = \sum_{i,j,k} (-1)^{\frac{j-i}{2}}A^j z^k \rk W^i_\jk.$$
It is an easy exercise to check that for a diagram $D$ of $L \subset \rpt$ it follows from the construction that $\chi(C_\bullet(D)) = \chi(H_\bullet(D)) = \kbsm(\rpt)(L)$, where the last equation holds by substituting $[x]$ for $z + z^{-1}$.

\begin{example}
In the link in Fig. \ref{fig:ex} we get 
$C_{00} = V_1 \wedge \bV_2 \shift{2}$,
$C_{01} = \bV$,
$C_{10} = \bV$ and 
$C_{11} = V_1 \wedge \bV_2 \shift{-2}$.
Omitting shifts, partial differentials work in the following manner: 
$d_{0\star}(1_1\wedge\bi_2) = \bi$,
$d_{0\star}(-1_1\wedge\bX_2) = \bX$,
$d_{0\star}(X_1\wedge\bi_2) = 0$,
$d_{0\star}(-X_1\wedge\bX_2) = 0$;
$d_{\star0}(\bi_2 \wedge 1_1) = \bi$,
$d_{\star0}(-\bi_2 \wedge X_1) = 0$,
$d_{\star0}(\bX_2 \wedge 1_1) = -\bX$,
$d_{\star0}(-\bX_2 \wedge X_1) = 0$;
$d_{\star1}(\bi) = X_1 \wedge \bi_2$,
$d_{\star1}(-\bX) = -X_1 \wedge \bX_2$;
$d_{1\star0}(\bi) = -\bi_2  \wedge X_1$,
$d_{1\star0}(\bX) = -\bX_2  \wedge X_1$.
The total differential is:
$d^{(2)}(1_1\wedge\bi_2) = \bi \oplus (-\bi)$,
$d^{(2)}(1_1\wedge\bX_2) = (-\bX) \oplus \bX$,
$d^{(2)}(X_1\wedge\bi_2) = 0 \oplus 0$,
$d^{(2)}(X_1\wedge\bX_2) = 0 \oplus 0$;
$d^{(0)}(\bi \oplus 0) = X_1 \wedge \bi_2$,
$d^{(0)}(\bX \oplus 0) = X_1 \wedge \bX_2$,
$d^{(0)}(0 \oplus \bi) = X_1 \wedge \bi_2$,
$d^{(0)}(0 \oplus \bX) = X_1 \wedge \bX_2$;
$d^{(-2)} = 0$.
It clearly holds that $d\circ d = 0$. The homology groups are
$H_2 \cong \langle X_1 \wedge \bi_2, X_1 \wedge \bX_2 \rangle \shift{2}$, 
$H_0 \cong 0$ and
$H_{-2} \cong \langle 1_1 \wedge \bi_2, 1_1 \wedge \bX_2 \rangle \shift{-2},$
i.e. the only non-trivial dimensions of $H_\bullet$ are $(H_2)_{4,1} \cong (H_2)_{4,-1} \cong (H_2)_{-4,1} \cong (H_2)_{-4,-1} \cong \Z$.
The Euler characteristic is $\chi(H_\bullet) = (-A^4 - A^{-4})(z+z^{-1}) = (-A^4 - A^{-4})[x]$.

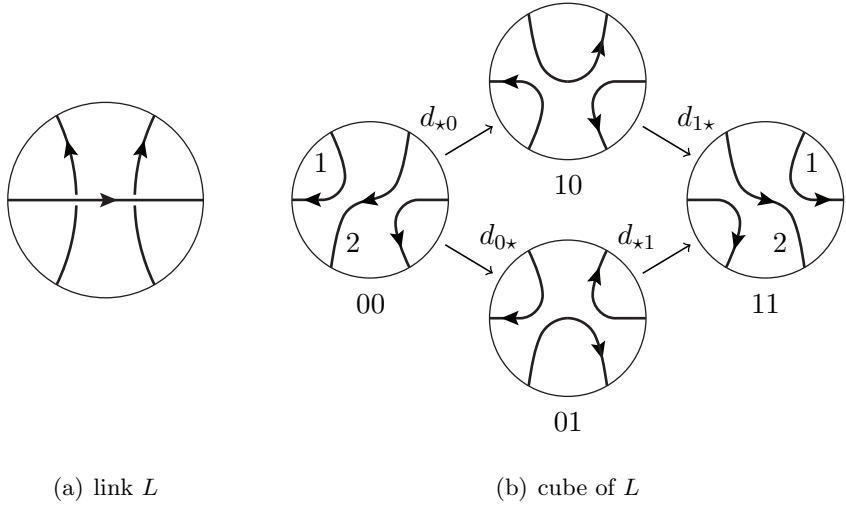
\begin{figure}[tbh]
\centering
\subfigure[link $L$]{
\raisebox{1.95cm}{
\includegraphics{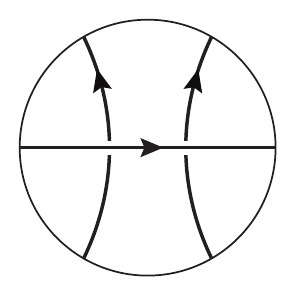}
\label{fig:example}
}
}
\subfigure[cube of $L$]{
\raisebox{0.5cm}{
\begin{tikzpicture}[scale=2.1]
\node[draw,shape=circle, minimum size=2.3cm, draw=white] (AA) at (0,0) {};
\node[draw,shape=circle, minimum size=2.3cm, draw=white] (BB) at (1.25,0.75) {};
\node[draw,shape=circle, minimum size=2.3cm, draw=white] (CC) at (1.25,-0.75) {};
\node[draw,shape=circle, minimum size=2.3cm, draw=white](DD) at (2.5,0) {};

\node (A) at (0,0) {\begin{overpic}[scale=1]{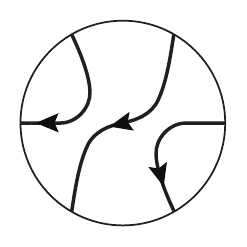}\put(20,65){$1$}\put(37,22){$2$}\put(42.5,-10){$00$}\end{overpic}};
\node (B) at (1.25,0.75) {\begin{overpic}{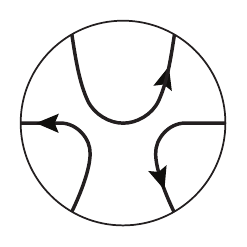}\put(42.5,-10){$10$}\end{overpic}};
\node (C) at (1.25,-0.75) {\begin{overpic}{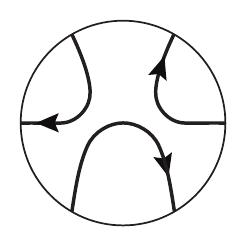}\put(42.5,-10){$01$}\end{overpic}};
\node (D) at (2.5,0) {\begin{overpic}{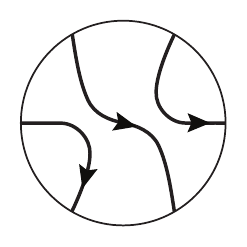}\put(71,65){$1$}\put(54,22){$2$}\put(42.5,-10){$11$}\end{overpic}};

\path[semithick,->]
(AA) edge node[above left]{$d_{\star 0}$} (BB)
(AA) edge node[above right]{$d_{0 \star}$} (CC)
(BB) edge node[above right]{$d_{1 \star}$} (DD)
(CC) edge node[above left]{$d_{\star 1}$} (DD);
\end{tikzpicture}
}
}
\caption{The link $L$ and its cube}
\label{fig:ex}
\end{figure}

\end{example}

\section{Proofs}
\label{sec:proofs}

Lemmas \ref{lemma:invariance-circle-order} - \ref{lemma:invariance-circle-orientation} show that the homology is independent of the free choices we made for choosing orderings of circles and orientations of the link and the circles in the states.

\begin{lemma}
The homology is invariant under the ordering of circles.
\label{lemma:invariance-circle-order}
\begin{proof}
For each state $s$ of $D$, let $o_s$ and $o_s'$ be two different orderings of circles in $s$ and let  $o_s$ and $o_s'$ differ by a permutation $\pi_s$. Permutations $\pi_s$ induce isomorphisms $P_{\pi_s}: C_s \lra C_s$ on the associated modules $C_s$. 
Let, for two adjacent states $q=\alpha_{\star \rightarrow 0}$ and $r=\alpha_{\star \rightarrow 1}$, the partial differential $d_\alpha: C_q \lra C_r$ be defined in terms of the first ordering as $P_{\rho_r} \circ d \circ P_{\sigma_q}$, where in place of $d$'s we have either $m\wedge \I$, $\Delta \wedge \I$ or $0$. 
It follows from definition that  the differential in terms of the second ordering equals $P_{\pi_r \rho_r} \circ d \circ P_{\sigma_q  \pi_q^{-1}}$.
Since $P_{\pi_r \rho_r} \circ d \circ P_{\sigma_q \pi_q^{-1}} \circ P_{\pi_q} = P_{\pi_r} \circ P_{\rho_r} \circ d \circ P_{\sigma_q}$, it follows that the maps $\lbrace P_{\pi_s} \rbrace_s$ form a chain map.
\end{proof}
\end{lemma}

\begin{lemma}
The homology is invariant under the change of link orientation.
\begin{proof}
It is enough to prove invariance under changing the orientation of one component.
Let $c$ be a crossing of the changed component.
If $c$ is a self-crossing, changing the orientation of both strands of $c$ has the same effect as transposing the orderings of the circles that split at $c$, which is invariant under lemma \ref{lemma:invariance-circle-order}.
If $c$ is not a self-crossing, changing the orientation of one component leads to inverting the locally consistent orientation on all participating circles and transposing the ordering of the two circles if the bifurcation site is of type $1 \rightarrow 2$; by writing down the partial differentials before and after the orientation change, it is easy to check that they both agree.
\end{proof}
\end{lemma}

\begin{lemma}
The homology is invariant under the change of orientations of the circles in the states.
\label{lemma:invariance-circle-orientation}
\begin{proof}

Let $f_s: C_s \lra C_s$ be the semi-identity that sends $X$ to $-X$ and $1$ to $1$ (resp. $\bX$ to $-\bX$ and $\bi$ to $\bi$) on the $V$  (resp. $\bV$) components of $C_s$ that represent circles with changed orientations. We observe that changing the orientation of a circle changes the local consistency in the definition of the partial differential $d$. But this is exactly the opposite of what $f$ does, hence $f \circ d \circ f^{-1} = d.$
\end{proof}
\end{lemma}

\begin{proof}[Proof of theorem \ref{thm:dd}.]

To show that $d \circ d=0$ holds for a projective link diagram $D$, it is sufficient to prove that every $2$-dimensional face of $\left\llbracket D \right\rrbracket$ is anticommutative. Every $2$-dimensional face corresponds to a smoothening of $n-2$ crossings, since the remaining two crossings are resolved in four different ways. Since the circles that are non-adjacent to either of the two crossings are mapped by $\I$, they can be omitted in the proofs. It therefore suffices to prove anticommutativity for all possible 2-crossing projective diagrams.
For constructing all such diagrams, we take a similar approach as in \cite{APS} and \cite{Ma}. Here is an outline.

Take all 4-valent graphs in $\rp^2$ with 2 rigid vertices and replace each vertex with either an overcrossing or an  undercrossing. This produces all possible 2-crossing diagrams.
Certain graphs can be omitted due to symmetries that leave circles in a natural 1-1 correspondence preserved under partial differentials.
Take vertices $v$ and $u$ and enumerate their edges by $v_0, v_1, v_2, v_3$ and $u_0, u_1, u_2, u_3$ (Fig. \ref{fig:rigid}).

\begin{figure}[tbh]
\centering
\begin{overpic}{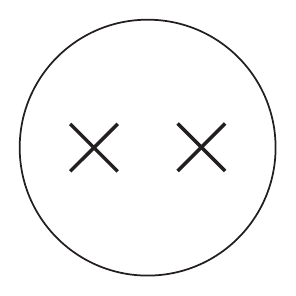} \put(15,63){$v_0$}  \put(36,63){$v_1$}  \put(52,63){$u_0$}  \put(73,63){$u_1$} 
\put(15,32){$v_3$}  \put(36,32){$v_2$}  \put(52,32){$u_3$}  \put(73,32){$u_2$} \end{overpic}
\caption{Vertices $v$ and $u$}
\label{fig:rigid}
\end{figure}
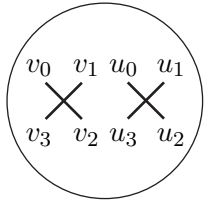

We disregard non-connected graphs, since it clearly holds that in this case either $d_{0\star} = -d_{1\star}, d_{\star0} = d_{\star 0} $ or $d_{0\star} = d_{1\star}, d_{\star0} = -d_{\star 0} $.
We may assume that $v_0$ is connected to $v_1$.
Two graphs are symmetric if they are related by the following operations (see \cite{APS} for details):
\begin{enumerate}
\item exchanging $v$ and $u$ produces the same graph;
\item a cyclic permutation $v_i \rightarrow v_{i-k\bmod 4}$ and $u_i \rightarrow u_{i-k  \bmod 4}$ for $i \in \lbrace 0,1,2,3 \rbrace$ if $v_k$ and $u_k$ are connected for some $k \in \lbrace 0,1,2,3 \rbrace$;
\item a flip: either $v_i \leftrightarrow v_{i+2 \bmod 4}$ or $u_i \leftrightarrow u_{i+2 \bmod 4}$ are exchanged for some $i \in \lbrace 0,1,2,3 \rbrace$.
\end{enumerate}

\begin{lemma} There are exactly 6 non-symmetric graphs in $\rp^2$. The graphs are represented in Fig. \ref{fig:essential}.
\label{lemma:sym-graphs}
\end{lemma}
\begin{figure}[tbh]
\centering
\subfigure[]{\includegraphics[]{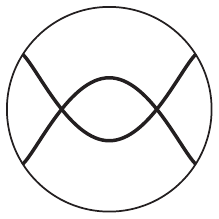}\label{fig:ea}}\;\;\;
\subfigure[]{\includegraphics[]{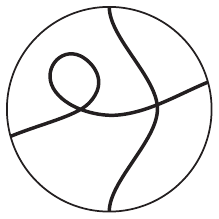}\label{fig:eb}}\;\;\;
\subfigure[]{\includegraphics[]{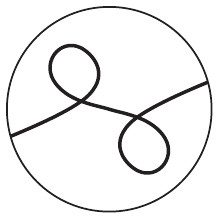}\label{fig:ec}} \\
\subfigure[]{\includegraphics[]{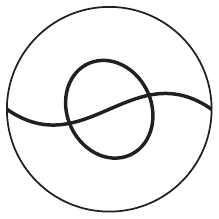}\label{fig:ed}}\;\;\;
\subfigure[]{\includegraphics[]{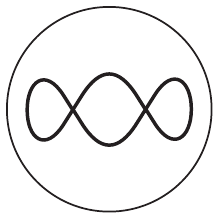}\label{fig:ee}}\;\;\;
\subfigure[]{\includegraphics[]{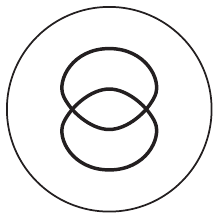}\label{fig:ef}}
\caption{Non-symmetric 4-valent connected graphs with 2 vertices in $\rp^2$}
\label{fig:essential}
\end{figure}

To list all the possible 2-crossing links, we place all 4 possible combinations of crossing on each vertices of the above graphs.
We can omit links where the $1\rightarrow 1$ bifurcation appears in each of the composites of the partial differentials.

\begin{lemma} Up to symmetry and disregarding non-essential cases with two  $1\rightarrow 1$ bifurcations appearing on opposite sides of the composites, there are 11 projective links with two crossings and are represented in Fig. \ref{fig:essentiallinks}.
\label{lemma:sym-crossing-links}
\end{lemma}

\begin{figure}[htb]
\centering
\subfigure[]{\includegraphics[]{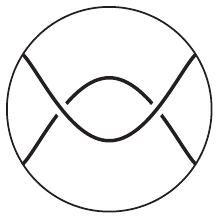}\label{fig:ea1}}
\subfigure[]{\includegraphics[]{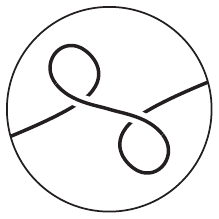}\label{fig:ec1}}
\subfigure[]{\includegraphics[]{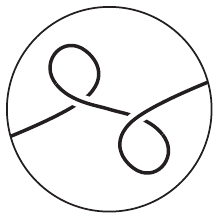}\label{fig:ec2}}
\subfigure[]{\includegraphics[]{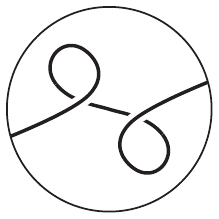}\label{fig:ec3}} \\
\subfigure[]{\includegraphics[]{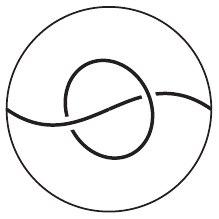}\label{fig:ed1}}
\subfigure[]{\includegraphics[]{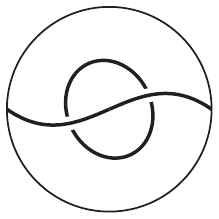}\label{fig:ed2}}
\subfigure[]{\includegraphics[]{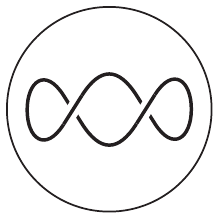}\label{fig:ee1}} 
\subfigure[]{\includegraphics[]{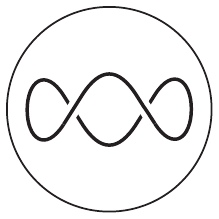}\label{fig:ee2}}\\
\subfigure[]{\includegraphics[]{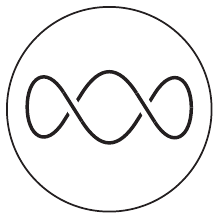}\label{fig:ee3}}
\subfigure[]{\includegraphics[]{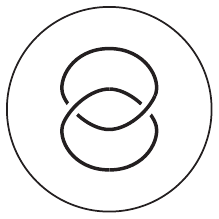}\label{fig:ef1}}
\subfigure[]{\includegraphics[]{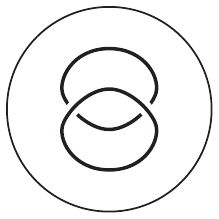}\label{fig:ef2}}
\caption{Non-symmetric essential links with 2 crossings in $\rp^2$}
\label{fig:essentiallinks}
\end{figure}

The proofs of lemmas \ref{lemma:sym-graphs} and \ref{lemma:sym-crossing-links} rely on case-by-case checking and are left to the reader.

\begin{figure}[htb]
\centering
\subfigure[link $L$]{
\raisebox{1.95cm}{
\includegraphics{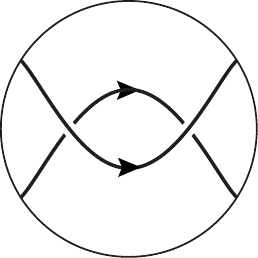}
}
}
\subfigure[cube of $L$]{
\raisebox{0.5cm}{
\begin{tikzpicture}[scale=2.1]
\node[draw,shape=circle, minimum size=2.3cm, draw=white] (AA) at (0,0) {};
\node[draw,shape=circle, minimum size=2.3cm, draw=white] (BB) at (1.25,0.75) {};
\node[draw,shape=circle, minimum size=2.3cm, draw=white] (CC) at (1.25,-0.75) {};
\node[draw,shape=circle, minimum size=2.3cm, draw=white](DD) at (2.5,0) {};

\node (A) at (0,0) {\begin{overpic}[scale=1]{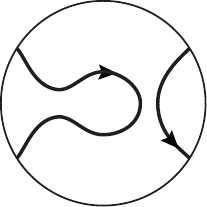}\put(42.5,-17){$00$}\end{overpic}};
\node (B) at (1.25,0.75) {\begin{overpic}{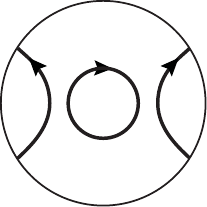}\put(42.5,-17){$10$} \put(13,44.5){$1$}\put(45.75,44.5){$2$} \end{overpic}};
\node (C) at (1.25,-0.75) {\begin{overpic}{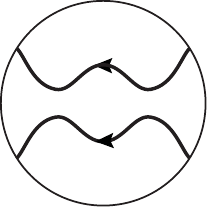}\put(42.5,-17){$01$}\end{overpic}};
\node (D) at (2.5,0) {\begin{overpic}{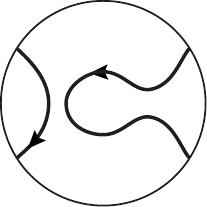}\put(42.5,-17){$11$}\end{overpic}};

\path[semithick,->]
(AA) edge node[above left]{$d_{\star 0}$} (BB)
(AA) edge node[above right]{$d_{0 \star}$} (CC)
(BB) edge node[above right]{$d_{1 \star}$} (DD)
(CC) edge node[above left]{$d_{\star 1}$} (DD);
\end{tikzpicture}
}
}
\caption{The link $L$ of Fig. \ref{fig:ea1} and its cube}
\label{fig:essential-example}
\end{figure}
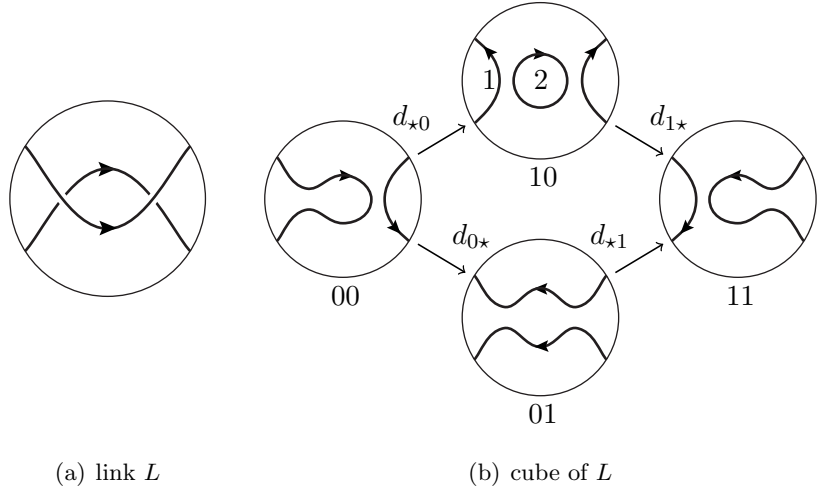

Due to lemmas \ref{lemma:invariance-circle-order} - \ref{lemma:invariance-circle-orientation} we can orient the above 11 links and their states arbitrarily, order the circles arbitrarily and calculate the differentials to show that $d\circ d=0$.
We check anticommutativity only for the link in Fig. \ref{fig:ea1}, the rest are done likewise.

From Fig.\;\ref{fig:essential-example} we get $d_{\star 0}(1) = 1_2 \wedge X_1 - X_2 \wedge 1_1$, $d_{\star 0}(-X) = - X_2 \wedge X_1$;
$d_{1 \star}(1_1 \wedge 1_2)=1$, $d_{1 \star}(-1_1 \wedge X_2)=X$, $d_{1 \star}(-X_1 \wedge 1_2)=X$, $d_{1 \star}(X_1 \wedge X_2)=0$; 
$d_{0 \star} = d_{\star 1} = 0$,
hence $d_{\star 0}d_{1 \star} = d_{0 \star}d_{\star 1} = 0$.
\end{proof}

\begin{proof}[Proof of theorem \ref{thm:invariance}]
The proofs of invariance under moves $\Rii$ and $\Riii$ almost entirely coincide with those of the classical case \cite{BN}.
The main tool used for showing invariance is the following cancelation principle:
\begin{lemma}
Let $C$ be a chain complex and let $C' \subset C$ be s sub-chain complex.  Then, if $C'$ is acyclic it holds that  $H(C) = H(C/C')$, on the other hand, if  $C/C'$ is acyclic it holds that $H(C) = H(C')$.
\begin{proof}
Both equalities follow trivially from the long exact homology sequence
$$\cdots \lra H^n(C') \lra H^n(C) \lra H^n(C/C') \lra H^{n+1}(C) \lra \cdots$$
associated with the short exact sequence $0 \lra C' \lra C \lra C/C' \lra 0.$
\end{proof}
\label{lemma:sub}
\end{lemma}

\begin{proof}[Invariance under $\Rii$]
\renewcommand{\qedsymbol}{}

The cube of $\bracsmaller{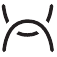}$ can be expressed in terms of subcubes as indicated on the left diagram below. 
\begin{center}
\begin{tikzpicture}[scale=1.8]
\node[draw,shape=circle, minimum size=1.4cm, draw=white] (AA) at (0,0) {}; \node[draw,shape=circle, minimum size=1.5cm, draw=white] (AB) at (1,0.75) {};
\node[draw,shape=circle, minimum size=1.5cm, draw=white] (BA) at (1,-0.75) {}; \node[draw,shape=circle, minimum size=1.7cm, draw=white](BB) at (2,0) {};
\node (A) at (0,0) {$\bracmid{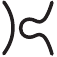}\{-2\}$}; \node (B) at (1,0.75) {$\bracmid{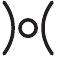}$};
\node (C) at (1,-0.75) {$\bracmid{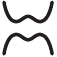}$}; \node (D) at (2 ,0) {$\bracmid{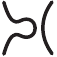}\{2\}$};
\path[semithick,->]
(AA) edge node[above left]{$\Delta$} (AB) (AA) edge node[below left]{} (BA) (AB) edge node[above right]{$m$} (BB) (BA) edge node[above left]{} (BB);
\node at (1,0) {$C$}; \node at (3.2,0) {$\supset$};
\node[draw,shape=circle, minimum size=1.0cm, draw=white] (AA) at (4,0) {}; \node[draw,shape=circle, minimum size=1.5cm, draw=white] (AB) at (5,0.75) {};
\node[draw,shape=circle, minimum size=1.0cm, draw=white] (BA) at (5,-0.75) {}; \node[draw,shape=circle, minimum size=1.7cm, draw=white](BB) at (6,0) {};
\node at (4,0) {$0$}; \node at (5,0.75) {$\bracmid{proof2a}_1$};
\node at (5,-0.75) {$0$}; \node at (6 ,0) {$\bracmid{proof2c}\{2\}$};
\path[semithick,->]
(AA) edge node[above left]{} (AB) (AA) edge node[above right]{} (BA) (AB) edge node[above right]{$m$} (BB) (BA) edge node[above left]{} (BB);
\node at (5,0) {$C'$};
\end{tikzpicture}
\end{center}
Let $C$ be the flattening of this cube. The chain complex $C$ contains a subcomplex $C'$ in which the subscript $1$ of $\bracsmaller{proof2a}_1$ denotes the cube of submodules where the trivial "middle" circle is not assigned the module $V$, but instead the free module $\langle 1 \rangle$. Since $m$ is an isomorphism in $C'$, $C'$ is acyclic and $H(C) \cong H(C/C')$ by lemma \ref{lemma:sub}.

\begin{center} 
\begin{tikzpicture}[scale=1.8]
\node[draw,shape=circle, minimum size=1.4cm, draw=white] (AA) at (0,-2.5) {}; \node[draw,shape=circle, minimum size=1.5cm, draw=white] (AB) at (1,-1.75) {};
\node[draw,shape=circle, minimum size=1.5cm, draw=white] (BA) at (1,-3.25) {}; \node[draw,shape=circle, minimum size=1.0cm, draw=white](BB) at (2,-2.5) {};
\node at (0,-2.5) {$\bracmid{proof2b}\{-2\}$}; \node at (1,-1.75) {$\bracmid{proof2a}_X$};
\node at (1,-3.25) {$\bracmid{proof2d}$}; \node at (2 ,-2.5) {$0$};
\path[semithick,->]
(AA) edge node[above left]{$\Delta$} (AB) (AA) edge node[below left]{$d_{\star 0}$} (BA)
(AB) edge node[above right]{} (BB) (BA) edge node[above left]{} (BB) (AB) edge node[right]{$\tau$} (BA);
\node at (-1.3,-2.5) {$C/C'$};
\end{tikzpicture}
\end{center}

The map $\Delta$ is an isomorphism in $C/C'$, so we can define a map $\tau$ that is the the composition $\tau = d_{\star 0} \Delta^{-1} $. 
Let $C''$ be the subcomplex of $C/C'$ consisting of all elements $\alpha \in \bracsmaller{proof2b}$ and all elements of the form $\beta \oplus \tau \beta \in \bracsmaller{proof2a}_X \oplus \bracsmaller{proof2d}$.
Taking $C/C'$ mod $C''$, we kill $\bracsmaller{proof2b}$ and impose the relation $\beta \oplus 0 = 0 \oplus \tau \beta$ in $\bracsmaller{proof2a}_X \oplus \bracsmaller{proof2d}$, but we have an arbitrary choice of $\gamma \in \bracsmaller{proof2d}$, hence $(C/C')/C''$ is isomorphic to $\bracsmaller{proof2d}$.

\end{proof}

\begin{proof}[Invariance under $\Riii$]
\renewcommand{\qedsymbol}{}

In order to prove invariance under the third Reidemeister move, we expand both cubes $\bracsmall{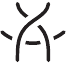}$ and $\bracsmall{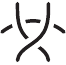}$ to the left and right diagrams below.
\begin{center} \begin{tikzpicture}[scale=1.7]
\node(AAA) at (0,0)  {$\brac{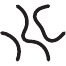}$}; \node(AAB) at (1,1) {$\brac{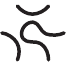}$}; 
\node(ABA) at (1,0) {$\brac{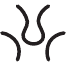}$};  \node(BAA)at (1,-1) {$\brac{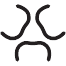}$};
\node(ABB) at (2,1) {$\brac{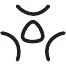}$}; \node(BAB)  at (2,0) {$\brac{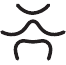}$};
\node(BBA)   at (2,-1) {$\brac{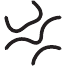}$}; \node(BBB) at (3,0) {$\brac{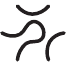}$};
\path[semithick,->]
(AAA) edge node{} (AAB) (AAA) edge node{} (ABA) (AAA) edge node{} (BAA)
(AAB) edge node[above]{$\Delta$} (ABB) (ABA) edge node{} (ABB) (ABA) edge node{} (BBA) (BAA) edge node{} (BBA) (ABB) edge node[above right]{$m$} (BBB)
(BAB) edge node{} (BBB) (BBA) edge node{} (BBB);
\path[-,white,line width=2 mm] (AAB) edge node{} (BAB) (BAA) edge node{} (BAB);
\path[semithick,->] (AAB) edge node{} (BAB) (BAA) edge node{} (BAB);
\node(AAA)  at (4,0) {$\brac{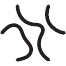}$}; \node(AAB)  at (5,1) {$\brac{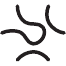}$}; 
\node(ABA)  at (5,0) {$\brac{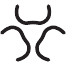}$}; \node(BAA)   at (5,-1) {$\brac{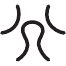}$};
\node(ABB)   at (6,1) {$\brac{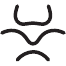}$};  \node(BAB)   at (6,0) {$\brac{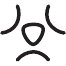}$}; 
\node(BBA)   at (6,-1) {$\brac{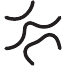}$}; \node(BBB)   at (7,0) {$\brac{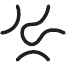}$};
\path[semithick,->]
(AAA) edge node{} (AAB) (AAA) edge node{} (ABA) (AAA) edge node{} (BAA) (AAB) edge node{} (ABB)
(ABA) edge node{} (ABB) (ABA) edge node{} (BBA) (BAA) edge node{} (BBA) (ABB) edge node{} (BBB)
(BAB) edge node[above]{$m$} (BBB) (BBA) edge node{} (BBB);
\path[-,white,line width=2 mm] (AAB) edge node{} (BAB) (BAA) edge node{} (BAB);
\path[semithick,->] (AAB) edge node[below=2pt, right=4pt]{$\Delta$} (BAB) (BAA) edge node{} (BAB);
\end{tikzpicture}
\end{center} 

In the above diagrams we omit degree shifts. For simplicity, we also keep the states of both cubes coherently oriented.
 
We repeat the definitions of modding out acyclic complexes $C'$ and $C''$ from the proof of $\Rii$ with the top-right parallelograms, as a result we obtain:
\begin{figure}[!h]
\centering
\begin{tikzpicture}[scale=1.7]
\node(AAA) at (0,0) {$\brac{proof3c8}$}; \node[inner sep=5pt] (AAB) at (1,1) {$0$}; 
\node(ABA) at (1,0) {$\brac{proof3b8}$};  \node(BAA) at (1,-1) {$\brac{proof3d1}$};
\node(ABB) at (2,1) {$\brac{proof3a1}_X$};  \node(BAB) at (2,0) {$\brac{proof3c4}$}; 
\node(BBA) at (2,-1) {$\brac{proof3c11}$}; \node[inner sep=5pt] (BBB) at (3,0) {$0$};
\path[semithick,->]
(AAA) edge node{} (AAB) (AAA) edge node{} (ABA) (AAA) edge node{} (BAA)
(AAB) edge node{} (ABB) (ABA) edge node[left=22pt,above=-13pt]{$d_{1,\star 01}$} (ABB) (ABA) edge node{} (BBA)
(BAA) edge node{} (BBA) (ABB) edge node{} (BBB) (BAB) edge node{} (BBB) (BBA) edge node{} (BBB);
\path[-,white,line width=2 mm] (AAB) edge node{} (BAB) (BAA) edge node{} (BAB);
\path[semithick,->] (AAB) edge node{} (BAB) (BAA) edge node[below=3pt,left=4pt]{$d_{1,\star 10}$} (BAB);
\path[semithick,->] (ABB) edge node[right]{$\tau_1$} (BAB);
\node (AAA) at (4,0) {$\brac{proof3c2}$}; \node[inner sep=5pt] (AAB) at (5,1) {$0$};
\node(ABA) at (5,0) {$\brac{proof3d2}$}; \node(BAA) at (5,-1) {$\brac{proof3b2}$};
\node(ABB) at (6,1) {$\brac{proof3c10}$}; \node(BAB) at (6,0) {$\brac{proof3a2}_X$}; 
\node(BBA) at (6,-1) {$\brac{proof3c5}$};\node[inner sep=5pt] (BBB) at (7,0) {$0$};
\path[semithick,->]
(AAA) edge node{} (AAB) (AAA) edge node{} (ABA) (AAA) edge node{} (BAA)
(AAB) edge node{} (ABB) (ABA) edge node[left=22pt,above=-13pt]{$d_{2,\star 01}$} (ABB) (ABA) edge node{} (BBA) (BAA) edge node{} (BBA)
(ABB) edge node{} (BBB) (BAB) edge node{} (BBB) (BBA) edge node{} (BBB);
\path[-,white,line width=2 mm] (AAB) edge node{} (BAB) (BAA) edge node{} (BAB);
\path[semithick,->] (AAB) edge node{} (BAB) (BAA) edge node[below=3pt,left=4pt]{$d_{2,\star 10}$} (BAB);
\path[semithick,->] (BAB) edge node[right]{$\tau_2$} (ABB);
\end{tikzpicture}
\end{figure}

Let $T$ be the homomorphism that sends the bottom-left parallelogram of the left cube to the bottom-left parallelogram of the right cube, but transposes the top-right parallelogram by 
sending $\bracsmall{proof3a1}_X  \oplus \bracsmall{proof3c4}$ to $\bracsmall{proof3a2}_X \oplus \bracsmall{proof3c10}$ via $T(\beta_1 \oplus \gamma_1) = \beta_2 \oplus \gamma_2$ for $\beta_1 \in \bracsmall{proof3a1}_X, \gamma_1 \in  \bracsmall{proof3c4}, \beta_2 \in \bracsmall{proof3a2}_X$ and $\gamma_2 \in \bracsmall{proof3c10}$.
Carefully writing down the definitions of the partial differentials, it easily follows that $\tau_1 \circ d_{1,\star 01} = d_{2,\star 01}$ and $d_{1,\star 10} = \tau_2 \circ d_{2,\star 10}$, hence, $T$ is an isomorphism.
\end{proof}
\begin{proof}[Invariance under $\Riv$ and $\Rv$]
\renewcommand{\qedsymbol}{}
Invariance under the Reidemeister moves $\Riv$ and $\Rv$ follows trivially, since there is a natural 1-1 correspondence that preserves the differentials between states of the diagram before and after these moves. 
\end{proof}
\end{proof}


\section{Conclusion and open questions}
\label{sec:conclusion}

This paper closes the chapter of categorifying the Kauffman bracket skein module of links in $I$-bundles over surfaces and perhaps opening a new chapter on categorifying Jones-like invariants of other $3$-manifolds, the most obvious candidates being the general lens space $L(p, q)$ and $S^1$-bundles over surfaces. 

Overcoming the difficulties of $1 \rightarrow 1$ bifurcations, the main difficulty from here on lies in the question of developing a non-recursive state-sum formula for describing the link in the particular skein module of the space. In particular, it seems that no easy-to-describe formula exists for the KBSM of $L(p, q)$ with $(p,q) \neq (2,1)$.


\section*{Acknowledgements} 
The author is grateful to M. Mroczkowski, M. Cencelj, J. Male\v si\v c and J. H. Przytycki for useful discussions and sharing their insights on the subject.


\end{document}